\documentclass[5p,times,twocolumn]{elsarticle}
\usepackage[utf8]{inputenc}
\usepackage{amsfonts}
\usepackage{amssymb}
\usepackage{mathtools}
\usepackage{amsthm}
\usepackage{enumerate}
\usepackage{appendix}
\usepackage{color}
\usepackage{lineno}

\newtheorem{thm}{Theorem}[section]

\newtheorem{lem}{Lemma}
\newtheorem{algo}[thm]{Algorithm}
\newtheorem{prop}[thm]{Proposition}
\theoremstyle{definition}
\newtheorem{defi}{Definition}[section]

\newtheorem{rmk}[defi]{Remark}

\newcommand{\RR}{\mathbb{R}}
\newcommand{\NN}{\mathbb{N}}

\newcommand{\diag}{\mathrm{diag}}

\begin{document}
\begin{frontmatter}

\title{Sequential convergence of AdaGrad algorithm for smooth convex optimization}

\author[1]{Cheik Traor\'e\corref{cor1}}
\author[2]{Edouard Pauwels}
\address[1]{ENS Paris-Saclay and MaLGa, DIMA, University of Genova}
\address[2]{IRIT, CNRS, Université Toulouse 3 Paul Sabatier}


\cortext[cor1]{Corresponding author \\ \emph{E-mail address:} cheik.traore@ens-paris-saclay.fr}

\begin{abstract}
    We prove that the iterates produced by, either the scalar step size variant, or the coordinatewise variant of AdaGrad algorithm, are convergent sequences when applied to convex objective functions with Lipschitz gradient. The key insight is to remark that such AdaGrad sequences satisfy a variable metric quasi-Fej\'er monotonicity property, which allows to prove convergence.
\end{abstract}

\end{frontmatter}

\section{Introduction}
We consider the problem of unconstrained minimization of a continuously differentiable convex function $F\colon \RR^n \mapsto \RR$ which gradient is globally Lipschitz. We will assume that the minimum of $F$ over $\RR^n$, $F^*$, is attained. We are interested in the sequential convergence of a largely used adaptive gradient method called AdaGrad.

\paragraph{Sequential convergence}
Continuous optimization algorithms are meant to converge if not to a global minimum at least to a local minimum of the cost function $F$, a necessary condition being, when the function is differentiable, Fermat rule, $\nabla F = 0$. 
Convergence of an iterative algorithm, producing a sequence of estimates in $\RR^n$, $(x_k)_{k \in \NN}$, can be measured in several ways: convergence of the norm of the gradients $\|\nabla F(x_k)\|_{k \in \NN}$, convergence of the suboptimality level $F(x_k) - F^*$, as $k$ grows to infinity. These measures of convergence do not translate directly into asymptotic convergence of the iterates $(x_k)_{k \in \NN}$ themselves. In general, this needs not be true, without additional assumptions. For example, when $F$ is strongly or strictly convex, since the minimum is uniquely attained, convergence of the gradient or the suboptimality level to $0$ implies convergence of the sequence.

Convergence of iterate sequences is an important measure of algorithmic stability. Indeed, in optimization applications (statistics \cite{bach2012optimization}, signal processing \cite{combettes2011proximal}) one may be concerned about the value of the argmin more than the minimum value. Sequential convergence ensures that the estimate of the argmin produced by the algorithm has some asymptotic stability property.

\paragraph{Adaptive gradient methods}
First order methods are the most widespread methods for machine learning and signal processing applications \cite{bottou2018optimization}. We will focus on  AdaGrad algorithm \cite{duchi2011adaptive}, which was initially developed in an online learning context, see also \cite{mcmahan2010adaptive}. This is a simple gradient method for which the step size is tuned automatically, in a coordinatewise fashion, based on previous gradient observations, this is where the term ``adaptive'' comes from. Interestingly, this adaptivity property found a large interest in training of deep networks \cite{goodfellow2016deep} with extensions and variants such as the widespread Adam algorithm \cite{DBLP:journals/corr/KingmaB14}. This success can be partially explained because adaptivity performs well in many applications, without requiring much manual tuning of step size decay. However this is not a consensus \cite{wilson2017marginal}.

Getting back to the convex world, it was suggested that adaptive step sizes give the possibility to use a single step size strategy, independent of the class of problem at hand: smooth versus nonsmooth, deterministic versus noisy \cite{levy2018online}. Indeed, it is known in convex optimization that constant step sizes can be used in the deterministic smooth case, while a decreasing schedule has to be used in the presence of nonsmoothness or noise. This idea was extended to adaptivity to strong convexity \cite{chen2018sadagrad} and its extensions \cite{xie2020linear}, as well as adaptivity in the context of variational inequalities \cite{bach2019universal}. 

In a more general nonconvex optimization context, there has been several recent attempts to develop a convergence theory for adaptive methods, with the application to deep network training in mind \cite{li2019convergence,malitsky2019adaptive,barakat2018convergence,ogaltsov2019adaptive,ward2019adagrad,defossez2020convergence,barakat2020convergence}. These provide qualitative convergence guaranties toward critical point or complexity estimates on the norm of the gradient, which are also, of course valid in the convex setting.

\paragraph{Fej\'er monotonicity and extensions}
In convex settings, the study of the convergence of the iterates of optimization algorithms has a long history. For many known first order algorithms, it turns out that the quantity $\|x_k - x^*\|_2^2$ is a Lyapunov function for the discrete dynamics for any solution $x^*$.  This property is called Fej\'er monotonicity, it allows to obtain convergence rates \cite{nesterov1998introductory} and also to prove convergence of iterate sequences in relation to Opial property. For example this property was used in \cite{combettes2000fejer}, to prove convergence of proximal point algorithm, forward-backward splitting method, Douglas-Rashford slitting method and more. 

One of the most important issues in studying AdaGrad is that it is not a descent algorithm as one is not guaranteed that a sufficient decrease condition would hold. Extension of Fej\'er monotonicity were proposed in order to handle such situations.  Quasi-Fej\'er monotonicity is the property of being Fej\'er monotone up to a summable error, its modern description was given in \cite{combettes2001quasi}. This can be used, to prove iterate convergence of algorithms such as block-iterative parallel algorithms, projected subgradient methods, stochastic subgradient method, perturbed optimization algorithms. Another issue related to AdaGrad is the fact that it induces a change of metric at each iteration, hence the notion of monotonicity which will be used is variable metric quasi-Fej\'er montonicity as introduced in \cite{combettes2013variable}.


\paragraph{Main result}

In this paper, we prove the sequential convergence of AdaGrad for smooth convex objectives. More precisely, we consider two versions of the algorithm, one with scalar step size on the one hand, and one with coordinatewise step size on the other hand. Both have been previously studied in the literature, but the second one is by far the most widely used in practice. Without a surprise the property of Fej\'er monotonicity is a central argument to prove this result. More precisely we show that sequences generated by AdaGrad are bounded (whenever the objective attains its minimum) and comply with variable metric quasi Fej\'er monotonicity, and our conclusion follows from the abstract analysis in \cite{combettes2013variable}. A crucial difference between our setting and the one described in \cite{duchi2011adaptive} is the fact that we do not require a bounded domain in Problem \eqref{eq:mainProblem}. Therefore, at least in the smooth convex case, our analysis shows that this assumption is not necessary.

\section{Technical preliminary}

\subsection{Notations}
In all this document we consider the set $\RR^n$ of real vectors of dimension $n$, $n \in \NN$. We denote by $x_i$ the $i^{th}$ component of the vector $x \in \RR^n$, with $i \in [1,2,\cdots,n]$. $\nabla F$ is the gradient of a differentiable function $F \colon \RR^n \rightarrow \RR$ and $\nabla_i F$ its $i^{th}$ component, corresponding to the $i$-th partial derivative. $\|\cdot\|$ is the Euclidean norm and $\langle \cdot, \cdot \rangle$ its associated scalar product. Let $(u_k)_{k \in \NN}$ a sequence in $\RR^n$. We denote by $u_{j,i}$ the $i^{th}$ component of the $j^{th}$ vector of the sequence $(u_k)_{k \in \NN}$. The diagonal matrix with entries of the vector $v \in \RR^n$ on its diagonal is represented by $\diag(v) \in \RR^{n \times n}$. We use the notation $\ell_{+}^{1}$ for the space of summable nonnegative sequences of real numbers. For a positive definite matrix $W \in \RR^{n \times n}$ we use the notation $\|d\|_W^2 = \left\langle Wd, d\right\rangle$ to denote the associated norm. We let $\succeq$ denote the partial order over symmetric matrices in $\RR^n$.


\subsection{Problem setting and assumptions}
\label{sec:assumptions}
Throughout this document we will consider the following unconstrained minimization problem 
\begin{align}
    \min_{x \in \RR^n} F(x)
    \label{eq:mainProblem}
\end{align}
where $F \colon \RR^n \mapsto \RR$ is differentiable and $n \in \NN$ is the ambient dimension. In addition, we assume that $F$ is convex and attains its minimum, that is, there exists $x^* \in \RR^n$ such that
\begin{align}
    \forall x \in \RR^n,\quad F(x) \geq F(x^*).
\end{align}
We finally assume that $F$ has as $L$-Lipschitz gradient, for some $L > 0$. More explicitly, $L$ is such that for any $x,y \in \RR^n$,
\begin{align}
    \|\nabla F(x) - \nabla F(y) \| \leq L\|x-y\|.
\end{align}
From this property, we can derive the classical Descent Lemma, which is a quantitative bound on the difference between $f$ and its first order Taylor expansion, see for example in \cite[Lemma 1.2.3]{nesterov1998introductory}.
\begin{lem}[Descent Lemma]\label{lem:desclemma} Suppose that $f \colon \RR^n \mapsto \RR$ has $L$-Lipschitz gradient. Then for all $x,y \in \RR^n$, we have
\begin{align}
    |f(y) - f(x) - \left\langle \nabla f(x), y - x \right\rangle| \leq \frac{L}{2} \|y - x\|^2.
\end{align}
\end{lem}
\subsection{Adaptive gradient algorithm (AdaGrad)}
\label{sec:adagrad}
We study two versions of AdaGrad, the original algorithm performing adaptive gradient steps at a coordinate level and a simplified version which uses a scalar step size. The latter variant, was coined as AdaGrad-Norm in \cite{ward2019adagrad}, it goes as follows:
\begin{algo}[AdaGrad-Norm]\label{algo:adagradSimple}
Given $x_0 \in \RR^n$, $v_0 = 0$, $\delta > 0$, iterate, for $k \in \NN$,
\begin{align}
    v_{k+1} &= v_k + \|\nabla F(x_k)\|^2 \nonumber \\
    x_{k+1} &= x_k - \frac{1}{\sqrt{v_{k+1} + \delta}} \nabla F(x_k) .
    \label{eq:adagradSimple}
\end{align}
\end{algo}
The original version presented in \cite{duchi2011adaptive} applies the same idea combined with coordinate-wise updates using partial derivatives. The algorithm is as follows.
\begin{algo}[AdaGrad]\label{algo:adagradSimplecoord}
Given $x_0 \in \RR^n$, $v_0 = 0$, $\delta > 0$, iterate, for $k \in \NN$ and $i \in [1, \cdots, n]$,
\begin{align}
    v_{k+1,i} &= v_{k,i} + \left(\nabla_i F(x_k)\right)^2 \nonumber \\
    x_{k+1,i} &= x_{k,i} - \frac{1}{\sqrt{v_{k+1,i} + \delta}}\nabla_i F(x_k),
    \label{eq:adagradSimplecoord}
\end{align}
\end{algo}
Our goal is, to prove that the sequences $(x_k)_{k \in \NN}$ generated by AdaGrad are convergent for both variants.
\begin{rmk}
In both algorithms, one could equivalently take $v_0 = \delta$ and run the recursions \eqref{eq:adagradSimple} and \eqref{eq:adagradSimplecoord} without $\delta$. However, we let $\delta$ appear explicitly to make clear that the denominator is non zero, enforced by the constraint $\delta > 0$. This constraint can be relaxed to $\delta \geq 0$ using the convention $\frac{0}{0} = 0$ while maintaining our sequential convergence. 
Indeed, if $\delta = 0$, for $k=0$, the denominator in \eqref{eq:adagradSimple} is equal to $0$ if and only if $\nabla F(x_0) = 0$ which is a very special case for which one can set $x_k=x_0$ for all $k$. A similar reasoning can be applied to the coordinatewise version \eqref{algo:adagradSimplecoord}, with a more tedious discussion about each coordinate. We stick to the strict positivity constraint for simplicity of exposition.
\end{rmk}

\section{Results}
Our main result is the following
\begin{thm}\label{thm:main}
    Let $F \colon \RR^n \mapsto \RR$ be convex with $L$-Lipschitz gradient and attain its minimum on $\RR^n$. Then any sequence $(x_k)_{k \in \NN}$ generated by AdaGrad-Norm (Algorithm \ref{algo:adagradSimple}) or AdaGrad (Algorithm \ref{algo:adagradSimplecoord}) converges to a global minimum of $F$ as $k$ groes to infinity. 
\end{thm}

The coming section is dedicated to exposition of the proof arguments for this result

\subsection{Variable metric quasi Fej\'er monotonicity}
The following definition is a simplification adapted from the more general exposition given in \cite{combettes2013variable}.
\begin{defi}
\label{def:variableMetric}
Let $\left(W_{k}\right)_{k \in \NN}$ be a sequence of symmetric matrices such that $W_k \succeq \alpha I_n, \forall k \in \NN$, for some $\alpha > 0$.
Let $C$ be a nonempty, closed and convex subset of $\RR^n,$ and let $\left(x_{k}\right)_{k \in \NN}$ be a sequence in $\RR^n$. Then $\left(x_{k}\right)_{k \in \NN}$ is variable metric quasi-Fej\'er monotone with respect to the target set $C$ relative to $\left(W_{k}\right)_{k \in \NN}$ if
\begin{align}
    &\left(\exists\left(\eta_{k}\right)_{k \in \NN} \in \ell_{+}^{1}(\NN)\right)(\forall z \in C)\left(\exists\left(\varepsilon_{k}\right)_{k \in \NN} \in \ell_{+}^{1}(\NN)\right)\nonumber \\
    &\left\|x_{k+1}-z\right\|^2_{W_{k+1}} \leqslant\left(1+\eta_{k}\right) \left\|x_{k}-z\right\|_{W_{k}}^2+\varepsilon_{k}, \quad (\forall k \in \NN). 
    \label{eq:variablemetrivfejer}
\end{align}
\end{defi}
For variable metric quasi-Fejér sequences the following proposition have already been established \cite[Proposition 3.2]{combettes2013variable}.
\begin{prop}\label{prop:quasifejerproperties}
Let $\left(u_{k}\right)_{k \in \NN}$ be a variable metric quasi-Fejér sequence relative to a nonempty, convex and closed set $C$ in $\RR^n$. These assertions hold.
\begin{enumerate}[(i)]
    \item \label{prop:quasifejerproperties1} $(\|u_k - u\|_{W_k})_{k \in \NN}$ converges for all $u \in C$.
    \item \label{prop:quasifejerproperties2} $\left(u_{k}\right)_{k \in \NN}$ is bounded.
\end{enumerate}
\end{prop}
The next theorem allows to prove sequential convergence of variable metric quasi Fej\'er sequences. This result is again due to \cite[Theorem 3.3]{combettes2013variable}.
\begin{thm}\label{thm:cnvgevarmet}
     Let $\left(W_{k}\right)_{k \in \NN}$ be a sequence of symmetric matrices such that $W_k \succeq \alpha I_n, \forall k \in \NN$, for some $\alpha > 0$. We suppose that the sequence $\left(W_{k}\right)_{k \in \NN}$ converges to $W$. Let $\left(x_{k}\right)_{k \in \NN}$ be a variable metric quasi-Fej\'er sequence with respect to a closed target set $C \subset \RR^n$. Then $\left(x_{k}\right)_{k \in \NN}$ converges to a point in $C$ if and only if every cluster point of $\left(x_{k}\right)_{k \in \NN}$ is in $C$.
\end{thm}

\begin{rmk}\label{rmk:quasifejer}
If, for a variable metric quasi-Fej\'er sequence, $\left(W_{k}\right)_{k \in \NN}$ is constant and $\left(\eta_{k}\right)_{k \in \NN}$ is null for all $k \in \NN$, the sequence is simply called a quasi-Fej\'er sequence. Moreover if $\left(\varepsilon_{k}\right)_{k \in \NN}$ is null for all $k \in \NN$, it is called a Fej\'er monotone sequence, which provides a Lyapunov function. Of course, for these two special cases, the results stated above hold.
\end{rmk}
\subsection{Convergence of AdaGrad-Norm}
To prove convergence of sequences generated by AdaGrad-Norm,  we start with the following lemmas. 
\begin{lem}\label{lem:quadratic}
Let $a,b \geq 0$. If for $Z \geq 0$, $\displaystyle \frac{Z}{\sqrt{Z+a}} \leq b$, then $\displaystyle Z \leq b^2 + b\sqrt{a}$.
\end{lem}
\begin{proof}
We have 
\begin{align*}
    Z^2 - b^2Z - b^2a \leq 0.
\end{align*}
For the equation of second order, $\Delta = b^4 + 4b^2a \geq 0$. We have two distinct real roots and the leading coefficient is positive, so for $Z$ to satisfy the above inequality, we should have
\begin{align*}
    Z &\leq \frac{b^2 + \sqrt{b^4 + 4b^2a}}{2} \\
      &\leq \frac{b^2 + \sqrt{b^4} + \sqrt{4b^2a}}{2} \\
      &= b^2 + b\sqrt{a}.
\end{align*}
\end{proof}

\begin{lem}
Under the hypothese of Theorem \ref{thm:main},
suppose that $(x_k)_{k \in \NN}$ is a sequence generated by Algorithm \ref{algo:adagradSimple}. Then we have that $\sum_{k=0}^{\infty}\|\nabla F(x_k)\|^2$ is finite.
\label{lem:lemmaconvergence}
\end{lem}
\begin{proof} \text{ This proof is inspired by the proof of Lemma 4.1 in \cite{ward2019adagrad}. }
Fix $x^* \in \RR^n$ such that $F(x^{*})=\inf _{x} F(x)>-\infty$. We split the proof into two cases.
\begin{itemize}
\item Suppose that there exits an index $k_{0} \in \NN$ such that $\sqrt{v_{k_0} + \delta} \geq L$. It follows using the descent Lemma \ref{lem:desclemma}, for any $j \geq 1$
\begin{align}
&F(x_{k_{0}+j})  \nonumber\\ 
\leq \,& F(x_{k_{0}+j-1}) + \langle \nabla F(x_{k_{0}+j-1}), x_{k_{0}+j} - x_{k_{0}+j-1}\rangle \nonumber\\
&+ \frac{L}{2} \|x_{k_{0}+j} - x_{k_{0}+j-1}\|^2 \nonumber \\
\label{eq:desclemma3} =&  F(x_{k_{0}+j-1})-\frac{1}{\sqrt{v_{k_{0}+j}+ \delta}}\left(1-\frac{L}{2 \sqrt{v_{k_{0}+j}+ \delta}}\right)\left\|\nabla F(x_{k_{0}+j-1})\right\|^{2} \\
\label{eq:qlcq1} \leq\,&  F(x_{k_{0}+j-1})-\frac{1}{2 \sqrt{v_{k_{0}+j}+ \delta}}\left\|\nabla F(x_{k_{0}+j-1})\right\|^{2} \\
\label{eq:qlcq2} \leq\,& F(x_{k_{0}})-\sum_{\ell=1}^{j} \frac{1}{2 \sqrt{v_{k_{0}+\ell }+ \delta}}\left\|\nabla F(x_{k_{0}-1+\ell})\right\|^{2}  \\
 \leq\,& F(x_{k_{0}})-\frac{1}{2  \sqrt{v_{k_{0}+j}+ \delta}} \sum_{\ell=1}^{j}\left\|\nabla F(x_{k_{0}+\ell-1})\right\|^{2}, \nonumber
\end{align}
where the transition from \eqref{eq:desclemma3} to \eqref{eq:qlcq1} is because $\sqrt{v_{k_0 + j} + \delta} \geq \sqrt{v_{k_0} + \delta} \geq L$, for all $j \geq 0$, and \eqref{eq:qlcq2} is a recursion.
Fix any $j \geq 1$, let $Z=\sum_{k=k_{0}}^{k_0+j-1}\left\|\nabla F(x_{k})\right\|^{2}$. It follows from the preceding inequality that
\begin{align*}
&2\left(F(x_{k_{0}})-F(x^{*})\right) \geq 2\left(F(x_{k_{0}})-F(x_{k_0+j})\right) \\
\geq\;& \frac{\sum_{k=k_{0}}^{k_0+j-1}\left\|\nabla F(x_{k})\right\|^{2}}{\sqrt{v_{k_0+j}+ \delta}} = \frac{Z}{\sqrt{Z+v_{k_{0}}+ \delta}}
\end{align*}
By Lemma \ref{lem:quadratic}, it follows
\begin{align}
&\sum_{k=k_{0}}^{k_0+j-1}\left\|\nabla F(x_{k})\right\|^{2} \nonumber\\ 
\leq\;& 4\left(F(x_{k_{0}})-F(x^{*})\right)^2 + 2(F(x_{k_{0}})-F(x^{*}))\sqrt{v_{k_{0}}+ \delta}.
\label{eq:equationboundfinaladagrad}
\end{align}
Since $j$ was arbitrary, one may take the limit $j \rightarrow \infty$ and we have
\[
\sum_{k=k_{0}}^{\infty}\left\|\nabla F(x_{k})\right\|^{2} < \infty.
\]
That means $\sum_{k=0}^{\infty}\left\|\nabla F(x_{k})\right\|^{2} < \infty$, which concludes the proof for this case.
\item On the contrary, we have that $\sqrt{v_{k}+\delta} < L$ for all $k \in \NN$, this means that $\forall k \in \NN$,
\begin{align*}
    \sum_{l=0}^{k} \|\nabla F(x_l)\|^2 &< L^2 - \delta.
\end{align*}
Letting $k$ goes to infinity gives
\begin{align*}
    \sum_{l=0}^{\infty} \|\nabla F(x_l)\|^2 &< L^2 - \delta < \infty,
\end{align*}
which is the desired result.
\end{itemize}
\end{proof}
We can now conclude the proof for AdaGrad-Norm.
\begin{proof}
Under the conditions of Theorem \ref{thm:main}, assume that $x_k$ is a sequence generated by AdaGrad-Norm.
Let $b_k = \sqrt{\delta + v_k}\geq \sqrt{\delta}$ for all $k \in \NN$, which is a non decreasing sequence. Let $x^* \in \arg \min F$ be arbitrary. By assumption $\arg \min F$ is nonempty and it is convex and closed since $F$ is convex and continuous. We have for all $k \in \NN$,
\begin{align*}
    &\|x_{k+1} - x^*\|^2\\
    =\;& \left\|x_k - x^* -\frac{1}{b_{k+1}}\nabla F(x_k)\right\|^2 \\
    =\;&\| x_k - x^* \|^2 + 2\left\langle \frac{1}{b_{k+1}}\nabla F(x_k), x^* - x_k \right\rangle + \frac{1}{b_{k+1}^2}\|\nabla F(x_k)\|^2.
\end{align*}
Thanks to the convexity of $F$, the above equality gives
\begin{align}\label{eq:forcvgcequasifejer}
    &\|x_{k+1} - x^*\|^2  \nonumber \\
    \leq \;&\| x_k - x^* \|^2 + \frac{2}{b_{k+1}} \left(F(x^*)-F(x_k)\right) + \frac{1}{b_{k+1}^2}\|\nabla F(x_k)\|^2 \nonumber \\
    \leq\;& \| x_k - x^* \|^2 + \frac{1}{\delta}\|\nabla F(x_k)\|^2.
\end{align}
By Lemma \ref{lem:lemmaconvergence}, $\|\nabla F(x_k)\|^2$ is summable. Hence $(x_k)_{k \in \NN}$ is a quasi-Fej\'er sequence relatively to $\arg \min F$. Proposition \ref{prop:quasifejerproperties} says that $(x_k)_{k \in \NN}$ is bounded. Thus it has an accumulation point. Then, thanks again to the Lemma \ref{lem:lemmaconvergence}, we have the set of accumulation points of $(x_k)_{k \in \NN}$ included in $\arg \min F$. So using Theorem \ref{thm:cnvgevarmet} and Remark \ref{rmk:quasifejer}, we conclude that $(x_k)_{k \in \NN}$ is convergent and that its limit is a global minimum of $F$.
\end{proof}
\subsection{Convergence of component-wise AdaGrad}
We now consider the case of AdaGrad in Algorithm \ref{algo:adagradSimplecoord}, taking into account the coordinatewise nature of the updates. The following corresponds to Lemma \ref{lem:lemmaconvergence} for this situation.
\begin{lem}
Under the hypothesis of Theorem \ref{thm:main},
suppose that $(x_k)_{k \in \NN}$ is a sequence generated by Algorithm \ref{algo:adagradSimplecoord}. We have that $\sum_{k=0}^{\infty}\|\nabla F(x_k)\|^2$ is finite.
\label{lem:lemmaconvergence2}
\end{lem}
\begin{proof} \text{ }

Let $I = \{i \in [1, \cdots, n]: \exists k_i \in \NN, \sqrt{v_{k_i,i} + \delta} \geq L\}$. Consider for each $i$ the smallest possible $k_i$ in the definition of $I$ and set $k_0 = \max k_i$, $i \in I$. If $I$ is empty, we have, $\forall k \in \NN$ and $\forall i \in [1, \cdots, n]$,
\begin{align*}
    \sum_{l=0}^{k} (\nabla_i F(x_l))^2 &< L^2 - \delta.
\end{align*}
Making $k$ goes to infinity gives, $\forall i \in [1, \cdots, n]$,
\begin{align*}
    \sum_{l=0}^{\infty}(\nabla_i F(x_l))^2 < L^2 - \delta < \infty \text{ and } \sum_{l=0}^{\infty}\|\nabla F(x_l)\|^2 < \infty.
\end{align*}
So let us assume that $I$ is not empty.
By Descent Lemma \ref{lem:desclemma}, for $j \geq 1$,
\begin{align} %
&F(x_{k_{0}+j}) \nonumber \\
 \leq\;&  F(x_{k_{0}+j-1}) + \langle \nabla F(x_{k_{0}+j-1}), x_{k_{0}+j} - x_{k_{0}+j-1}\rangle \nonumber \\
 &+ \frac{L}{2} \|x_{k_{0}+j} - x_{k_{0}+j-1}\|^2 \nonumber \\
=\;&  F(x_{k_{0}+j-1}) + \sum_{i=1}^{n} \nabla_i F(x_{k_{0}+j-1})(x_{k_{0}+j} - x_{k_{0}+j-1})_i \nonumber \\
&+ \frac{L}{2} \sum_{i=1}^{n} (x_{k_{0}+j} - x_{k_{0}+j-1})_i^2 \nonumber \\
=\;&  F(x_{k_{0}+j-1}) - \sum_{i=1}^{n} \frac{1}{\sqrt{v_{k_{0}+j, i} + \delta}} \left(\nabla_i F(x_{k_{0}+j-1})\right)^2 \nonumber \\
&+ \frac{L}{2} \sum_{i=1}^{n} \frac{1}{v_{k_{0}+j, i} + \delta} \left(\nabla_i F(x_{k_{0}+j-1})\right)^2 \nonumber \\
=\;& F(x_{k_{0}+j-1})-\sum_{i \in I}\frac{1}{\sqrt{v_{k_{0}+j, i} + \delta}}\left(1-\frac{L}{2( \sqrt{v_{k_{0}+j,i}+ \delta})}\right)\left(\nabla_i F(x_{k_{0}+j-1})\right)^{2} \nonumber \\ &\quad - \sum_{i \notin I}\frac{1}{\sqrt{v_{k_{0}+j, i} + \delta}}\left(1-\frac{L}{2( \sqrt{v_{k_{0}+j,i}+ \delta})}\right)\left(\nabla_i F(x_{k_{0}+j-1})\right)^{2}.
\label{eq:descentlemmacmp}
\end{align}
We will take care of the two sums separately. For $i \notin I$, $\sqrt{v_{k,i} + \delta} < L$ for all $k \in \NN$. Therefore, for all $k \in \NN$ and $i \notin I$, 
\begin{align}
    \sum_{\ell=0}^{k} \left(\nabla_{i} F(x_{\ell -1})\right)^{2} = v_{k,i} < L^2-\delta
    \label{eq:sumGradSquareFiniteInI}
\end{align}
Furthermore, we have for all $k \in \NN$
\begin{align}
  & -\sum_{l = 0}^k \sum_{i \notin I}\frac{1}{\sqrt{v_{l, i} + \delta}}\left(1-\frac{L}{2 \sqrt{v_{l,i}+ \delta}}\right)\left(\nabla_i F(x_{l-1})\right)^{2}\nonumber \\
  \leq& \frac{1}{\sqrt{\delta}} \left(\frac{L}{2 \sqrt{\delta}} \right) \sum_{l =0}^k \sum_{i \notin I} \left(\nabla_i F(x_{l-1})\right)^{2} \nonumber\\
  =&\frac{L}{2\delta} \sum_{i \notin I} \sum_{l =0}^k \left(\nabla_i F(x_{l-1})\right)^{2} \nonumber\\
  = &\frac{L}{2\delta}n(L^2-\delta) = C < + \infty,
  \label{eq:temp0}
\end{align}
where we let $C = \frac{L}{2\delta}n(L^2-\delta)$. This takes care of the first case.

Consider now the first sum in \eqref{eq:descentlemmacmp}. Since for $i \in I$, $j \geq 1$, $\displaystyle 1-\frac{L}{2\sqrt{v_{k_{0}+j,i}+ \delta}} \geq 1/2$, we have 
\begin{align}
    \label{eq:temp1}
    -\frac{1}{\sqrt{v_{k_{0}+j, i} + \delta}}\left(1-\frac{L}{2( \sqrt{v_{k_{0}+j,i}+ \delta})}\right)&\leq -\frac{1}{2}\frac{1}{\sqrt{v_{k_{0}+j, i} + \delta}}.
\end{align}
By recurrence on \eqref{eq:descentlemmacmp}, using \eqref{eq:temp1} and \eqref{eq:temp0}, it follows that for all $j \geq 1$,
\begin{align*}
    F(x_{k_{0}+j}) \,\leq& F(x_{k_{0}})-\frac{1}{2}\sum_{\ell=1}^{j} \sum_{i \in I} \frac{1}{ \sqrt{v_{k_{0}+\ell,i} + \delta}}\left(\nabla_i F(x_{k_{0}+\ell-1})\right)^{2} + C
\end{align*}
That is equivalent to
\begin{align*}
2\left(F(x_{k_{0}}) - F(x_{k_{0}+j}) + C \right) & \geq \sum_{\ell=1}^{j} \sum_{i \in I} \frac{1}{ \sqrt{v_{k_{0}+\ell,i} + \delta}}\left(\nabla_i F(x_{k_{0}+\ell-1})\right)^{2} \nonumber \\
    &\geq \sum_{i \in I} \frac{1}{\sqrt{v_{k_{0}+j,i}+ \delta}} \sum_{\ell=1}^{j}\left(\nabla_i F(x_{k_{0}+\ell-1})\right)^{2}
\end{align*}
Fix $p \in I$, we deduce that for all $j \geq 1$,
\begin{align*}
    &\frac{1}{\sqrt{v_{k_{0}+j,p}+ \delta}} \sum_{\ell=1}^{j}\left(\nabla_p F(x_{k_{0}+\ell-1})\right)^{2} \\
    \leq\;& \sum_{i \in I} \frac{1}{\sqrt{v_{k_{0}+j,i}+ \delta}} \sum_{\ell=1}^{j}\left(\nabla_i F(x_{k_{0}+\ell-1})\right)^{2} \\
    \leq\;& 2\left(F(x_{k_{0}}) - F(x_{k_{0}+j}) +C \right) \\
    \leq\;& 2\left(F(x_{k_{0}}) - F(x^*) + C\right).
\end{align*}
Fix any $j \geq 1$, let $Z=\sum_{k=k_{0}}^{k_0+j-1}\left(\nabla_p F(x_{k})\right)^{2}$. We have $v_{k_{0}+j,p} = Z + v_{k_{0},p}$ and the previous inequality reads
\begin{align*}
    \frac{Z}{\sqrt{Z+v_{k_{0},p}+ \delta}} &\leq 2\left(F(x_{k_{0}}) - F(x^*) + C \right).
\end{align*}
By Lemma \ref{lem:quadratic}, we get
\begin{align}
    &\sum_{k=k_{0}}^{k_0+j-1}\left(\nabla_p F(x_{k})\right)^{2} \nonumber\\
    \leq\;& 4\left(F(x_{k_{0}})-F(x^{*})+ C\right)^2 + 2(F(x_{k_{0}})-F(x^{*})+C)\sqrt{v_{k_{0},p}+ \delta}.
\label{eq:equationboundfinaladagrad2}
\end{align}
We may let $j$ go to infinity and we obtain,
\begin{align*}
    \sum_{k=0}^{+\infty}\left(\nabla_p F(x_{k})\right)^{2} < \infty
\end{align*}
Since $p \in I$ was arbitrary, combining with \eqref{eq:sumGradSquareFiniteInI}, for all $i \in [1,\cdots,n]$
\begin{align*}
    \sum_{k=0}^{+\infty}\left(\nabla_i F(x_{k})\right)^{2} < \infty,
\end{align*}
and the result follows by summation
\begin{align*}
    \sum_{k=0}^{+\infty}\left\|\nabla F(x_{k})\right\|^{2} < \infty.
\end{align*}
\end{proof}
We conclude this section with the convergence proof for AdaGrad.
\begin{proof}
Under the conditions of Theorem \ref{thm:main}, assume that $x_k$ is a sequence generated by Algorithm \ref{algo:adagradSimplecoord}.
Let $b_{k,i} = \sqrt{\delta + v_{k,i}}$ for $k \in \NN$, $i \in [1,\cdots,n]$, all of them are increasing sequences. Fix any $x^* \in \arg\min F$, which is nonempty closed and convex since $F$ is convex, continuous and attains its minimum. Let $b_k = \left(b_{k,1}, \cdots, b_{k,n}\right) \in \RR^n$. We have for all $k \in \NN$ and $i = 1, \ldots, n$,
\begin{align*}
    &b_{k+1,i}\left(x_{k+1,i} - x_i^*\right)^2 \\
    =\;& b_{k+1,i}\left(x_{k,i} - x_i^* -\frac{1}{b_{k+1,i}}\nabla_i F(x_k)\right)^2 \\
    =\;& b_{k+1,i}\left( x_k - x^*\right)_i^2 + 2\left(\nabla_i F(x_k)\right)\left(x^* - x_k\right)_i + \frac{1}{b_{k+1,i}}\left(\nabla_i F(x_k)\right)^2.
\end{align*}
By summing over $i = 1,\ldots, n$, we get for all $k \in \NN$,
\begin{align*}
    &\sum_{i=1}^{n} b_{k+1,i}\left(x_{k+1} - x^*\right)_i^2 \\
    =\;& \sum_{i=1}^{n} b_{k+1,i}\left( x_k - x^*\right)_i^2 + 2\sum_{i=1}^{n} \left(\nabla_i F(x_k)\right)\left(x^* - x_k\right)_i\\
    &+ \sum_{i=1}^{n} \frac{1}{b_{k+1,i}}\left(\nabla_i F(x_k)\right)^2,
\end{align*}
and hence,
\begin{align*}
    &\|x_{k+1} - x^*\|_{B_{k+1}}^2 \\
    \leq\;& \sum_{i=1}^{n} b_{k+1,i}\left( x_k - x^*\right)_i^2 + 2\langle\nabla F(x_k), x^* - x_k\rangle + \frac{1}{\sqrt{\delta}}\|\nabla F(x_k)\|^2,
\end{align*}
where $B_{k+1}=\text{Diag}(b_{k+1}) \in \RR^{n \times n}$. Thanks to the convexity of $F$, the above inequality gives for all $k \in \NN$
\begin{align*}
    &\|x_{k+1} - x^*\|_{B_{k+1}}^2 \\
    \leq\;& \sum_{i=1}^{n} b_{k+1,i}\left( x_k - x^*\right)_i^2 + 2\left(F(x^*)-F(x_k)\right) + \frac{1}{\sqrt{\delta}}\|\nabla F(x_k)\|^2 \\
    \leq\;& \sum_{i=1}^{n} b_{k+1,i}\left( x_k - x^*\right)_i^2  + \frac{1}{\sqrt{\delta}}\|\nabla F(x_k)\|^2.
\end{align*}
It follows, for all $k \in \NN$,
\begin{align*}
    &\|x_{k+1} - x^*\|_{B_{k+1}}^2 - \frac{1}{\sqrt{\delta}}\|\nabla F(x_k)\|^2 \\
    \leq\;& \sum_{i=1}^{n} b_{k,i}\left( x_k - x^*\right)_i^2\frac{b_{k+1,i}}{b_{k,i}}  \\
    \leq\;& \max_{i \in [1,\cdots,n]}\frac{b_{k+1,i}}{b_{k,i}}\sum_{i=1}^{n} b_{k,i}\left( x_k - x^*\right)_i^2  \\
    =\;& \left(1+\left(\max_{i \in [1,\cdots,n]}\frac{b_{k+1,i}}{b_{k,i}}-1\right)\right)\sum_{i=1}^{n} b_{k,i}\left( x_k - x^*\right)_i^2 \\
    =\;& \left(1+\left(\max_{i \in [1,\cdots,n]}\frac{b_{k+1,i}}{b_{k,i}}-1\right)\right)\|x_{k} - x^*\|_{B_{k}}^2.
\end{align*}
Let $M \in \NN$, $M \geq 1$. For all $i \in [1, \cdots, n]$, we have, 
\begin{align*}
    \sum_{k=0}^{M-1}\frac{b_{k+1,i}}{b_{k,i}}-1 &= \sum_{k=0}^{M-1}\frac{b_{k+1,i}-b_{k,i}}{b_{k,i}} \\
    &\leq \sum_{k=0}^{M-1}\frac{b_{k+1,i}-b_{k,i}}{\sqrt{\delta}} \\
    & = \frac{1}{\sqrt{\delta}}b_{M,i} \\
    &< \infty,
\end{align*}
where the boundedness follows from Lemma \ref{lem:lemmaconvergence2}.
So $\forall i \in [1, \cdots, n]$, the sequence $(\frac{b_{k+1,i}}{b_{k,i}}-1)_{k \in \NN}$ is summable, and since $(b_{k,i})_{k \in \NN}$ is nondecreasing, it is also nonnegative. In particular the sequence $(\max_{i \in [1,\cdots,n]}\frac{b_{k+1,i}}{b_{k,i}}-1)_{k \in \NN}$ is summable and nonnegative.

Therefore $(x_k)_{k \in \NN}$ is variable metric quasi-Fej\'er  with target set $C = \arg\min F$, metric $W_k = B_k \succeq \frac{1}{\delta} I$, $\eta_k = \max_{i \in [1,\cdots,n]}\frac{b_{k+1,i}}{b_{k,i}}-1$ and $\epsilon_k = \frac{1}{\sqrt{\delta}}\|\nabla F(x_k)\|^2$, for all $k \in \NN$, using the notations of Definition \ref{def:variableMetric}. Note that $(\eta_k)_{k \in \NN}$ does not depend on the choice of $x^* \in C$ and is summable, and $(\epsilon_k)_{k \in \NN}$ is also summable by Lemma \ref{lem:lemmaconvergence2}, so that the definition applies. By Lemma \ref{lem:lemmaconvergence2}, $C$ contains all the cluster points and $(W_k)_{k \in \NN}$ converges. Thus Theorem \ref{thm:cnvgevarmet} allows us to conclude that $(x_k)_{k \in \NN}$ converges to a global minimum.
\end{proof}
\section{Discussion and future work}
Sequential convergence of AdaGrad in the smooth convex case constitutes a further adaptivity property for this algorithm. Fej\'er monotonicity plays an important role here as one would expect. It is interesting to remark that our analysis does not require any assumption on the objective $F$ beyond its Lipschitz gradient and the fact that it attains its minimum. Those are sufficient to ensure boundedness and convergence of any sequence. This is in contrast with analyses in more advanced, nonconvex, noisy settings where additional assumptions are required \cite{ward2019adagrad,defossez2020convergence}. Extensions of this analysis include the addition of noise or nonsmoothness in the convex case. It would also be interesting to see if the proposed approach allows to obtain better convergence bounds than the original regret analysis \cite{duchi2011adaptive,levy2018online}. 

A natural extension of the algorithm would take into account constrained versions of Problem \eqref{eq:mainProblem} by adding a projection step to Algorithm \eqref{eq:adagradSimplecoord}. Our analysis does not directly apply to such an algorithm because Lemma \ref{lem:lemmaconvergence2} would not be guaranteed to hold true, since in constraint optimization, the gradient of the objective may not vanish at the optimum. Therefore, the metric underlying the corresponding recursion would not stabilize and Theorem  \ref{thm:cnvgevarmet} would not apply. It is a topic of future research to describe a converging forward-backward, or projected, variant of the coordinatewise version of Adagrad. 

Finally, it would be interesting to relax the global Lipschicity assumption on the gradient to local Lipschicity, for example twice differentiability as proposed in \cite{malitsky2019adaptive}. We conjecture that our sequential convergence result still holds for convex objectives $F$ with full domain and locally Lipschitz gradient. The full domain assumption is crucial here since as we describe in next section , there exists a convex function $f \colon (-1/2,1/2) \mapsto \RR$, differentiable on its domain with locally Lipschitz gradient and a corresponding Adagrad sequence which diverges.

\subsection{On the importance of full domain assumption}
 \label{section_importance}
Do Algorithm \eqref{algo:adagradSimple} or \eqref{algo:adagradSimplecoord} converge when the global Lipschicity assumption is relaxed to local Lipschicity? We argue that full domain assumption would be essential for such a result. This section provides a counter example in dimension 1 of a convex function on $(-1/2,1/2)$ with locally Lipschitz derivative and an Adagrad sequence which is divergent. Note that in dimension 1, both algorithms are the same.

\begin{lem}
    \label{lem:interpolationDim1}
    Let $(x_k)_{k\in \NN}$ and $(z_k)_{k\in \NN}$ be two real sequence such that
    \begin{itemize}
        \item $x_1 < x_0$ and $z_1 \leq z_0$.
        \item $(x_{2k})_{k\in \NN}$ is strictly increasing and $(z_{2k})_{k\in \NN}$ is non decreasing.
        \item $(x_{2k+1})_{k\in \NN}$ is strictly decreasing and $(z_{2k+1})_{k\in \NN}$ is non increasing.
    \end{itemize}
    Then there exists a differentiable convex function $f \colon (\inf_{k \in \NN} x_k, \sup_{k \in \NN} x_k) \to \RR$ with locally Lipschitz gradient such that for all $k \in \NN$, $f'(x_k) = z_k$.
\end{lem}   
\begin{proof}
    For each $k > 1$, define
    \begin{itemize}
        \item $g_k \colon [x_0, x_{2k}] \to \RR$ to be the affine interpolant such that $g_k(x_{2i}) = z_{2i}$ for all $i = 0,\ldots k$.
        \item $h_k \colon [x_{2k+1},x_0] \to \RR$ to be the affine interpolant such that $h_k(x_{2i+1}) = z_{2i+1}$ for all $i = 0,\ldots k$ and $h_k(x_0) = z_0$.
        \item $f_k \colon [x_{2k+1},x_{2k}] \to \RR$ be such that $f_k(x) = \int_{x_0}^x g_k(t) dt$ for $x \geq x_0$ and $f_k(x) = \int_{x_0}^x h_k(t) dt$ for $x < x_0$.
    \end{itemize}
    $f_k$ is continuously differentiable on $(x_{2k+1},x_{2k})$ and we have for all $x \in (x_{2k+1},x_{2k})$ that $f_k'(x) = g_k(x)$ if $x \geq x_0$ and $f'_k(x) = h_k(x)$ if $x < x_0$. Its derivative is non decreasing and hence $f_k$ is convex and its derivative is locally Lipschitz because it is piecewise affine.
    
    For $k' \geq k$, the function $f_{k'}$ defined on $[x_{2k'+1},x_{2k'}]$ agrees with $f_k$ on $[x_{2k+1},x_{2k}] \subset [x_{2k'+1},x_{2k'}]$. The desired function $f$ is defined for any $x \in (\inf_{k \in \NN} x_k, \sup_{k \in \NN} x_k)$ to be the equal to $f_k(x)$ for any $k$ such that $x \in (x_{2k+1},x_{2k})$, such a $k$ must exist by properties of infimum and suppremum. This is the desired function $f$.
\end{proof}
We need a last technical lemma.
\begin{lem}
    Let $0<b<1$, $\delta>0$ and $c>0$. Setting
    \begin{align*}
        a = \pm \frac{b}{\sqrt{1 - b^2}} \sqrt{\delta + c},
    \end{align*}
    we have
    \begin{align}
        \frac{a^2}{\delta + c + a^2} = b^2.
    \end{align}
    \label{lem:technicalCounterEx}
\end{lem}
\begin{proof}
    We have
    \begin{align*}
        a^2 &= \frac{b^2}{1 - b^2} (\delta + c)\\
        &=\frac{b^2}{1 - b^2} (\delta + c + a^2) - a^2 \frac{b^2}{1 - b^2}.
    \end{align*}
    We deduce that
        \begin{align*}
        \frac{a^2}{1 - b^2} &=\frac{b^2}{1 - b^2} (\delta + c + a^2),
    \end{align*}
    and the desired result follows
\end{proof}

Consider the following sequence
\begin{itemize}
    \item $x_{2k} = \frac{1}{2} - \frac{1}{2^{k+1}}$, for all $k \in \NN$.
    \item $x_{2k+1} = \frac{-1}{2} + \frac{1}{2^{k+2}}$, for all $k \in \NN$.
\end{itemize}
This sequence complies with the hypotheses of Lemma \ref{lem:interpolationDim1}.  Furthermore, $|x_k - x_{k+1}|$ is an increasing sequence, strictly smaller than $1$. Fix $\delta = 15$ and $z_0 = 1$. This ensures that $x_0 - z_0 / \sqrt{\delta + z_0^2} = x_1$. This relation will be preserved by recursion, setting for all $k \in \NN$
\begin{itemize}
    \item $z_{2k+1} = - \frac{|x_{2k+2} - x_{2k+1}|}{\sqrt{1 - |x_{2k+2} - x_{2k+1}|^2}}  \sqrt{\delta + \sum_{i=0}^{2k} z_i^2}$.
    \item $z_{2k+2} =   \frac{|x_{2k+3} - x_{2k+2}|}{\sqrt{1 - |x_{2k+3} - x_{2k+2}|^2}}  \sqrt{\delta + \sum_{i=0}^{2k+1} z_i^2}$.
\end{itemize}
We have for all $k \in \NN$,
\begin{align*}
    \frac{|z_{2k+1}|}{\sqrt{\delta + \sum_{i=0}^{2k+1}z_i^2}} = |x_{2k+2} - x_{2k+1}|,
\end{align*}
by using Lemma \ref{lem:technicalCounterEx} with $a = |z_{2k+1}|$, $b =  |x_{2k+2} - x_{2k+1}|$, $c =\sum_{i=0}^{2k} z_i^2$ and the chosen value of $\delta$. Similarly, for all $k \in \NN$, by using Lemma \ref{lem:technicalCounterEx} with $a = |z_{2k+2}|$, $b =  |x_{2k+3} - x_{2k+2}|$, $c =\sum_{i=0}^{2k+1} z_i^2$, 
\begin{align*}
    \frac{|z_{2k+2}|}{\sqrt{\delta + \sum_{i=0}^{2k+2} z_i^2}} = |x_{2k+3} - x_{2k+2}|,
\end{align*}
In summary, we have for all $k \in \NN$,
\begin{align}
    \label{eq:tempInterpol1}
    \frac{|z_{k}|}{\sqrt{\delta + \sum_{i=0}^{k}z_i^2}} = |x_{k+1} - x_{k}|.
\end{align}
Carefully studying the signs of each sequence, we have for all $k \in \NN$
\begin{align*}
    \frac{z_{2k}}{\sqrt{\delta + \sum_{i=0}^{2k}} z_i^2} = x_{2k} - x_{2k+1} & \geq 0\\
    \frac{z_{2k+1}}{\sqrt{\delta +\sum_{i=0}^{2k+1}} z_i^2} = x_{2k+1} - x_{2k+2} & \leq 0
\end{align*}
and finally for all $k \in \NN$
\begin{align}
    \label{eq:tempInterpol2}
    x_k - \frac{z_{k}}{\sqrt{\delta + \sum_{i=0}^{k}} z_i^2} = x_{k+1}, 
\end{align}
so that the two sequences comply with recursion \eqref{eq:adagradSimple} with $z_k$ in place of derivatives. Let us check that the sequence $(z_k)_{k \in \NN}$ complies with Lemma \ref{lem:interpolationDim1}, which would conclude the proof. Given the alternating sign pattern of the sequence, it is sufficient to show that $|z_k|_{k \in \NN}$ is non decreasing. From \eqref{eq:tempInterpol1} and the fact that $|x_{k+1} - x_{k}|$ is increasing with $k$, we have for all $k \in \NN$
\begin{align*}
    \frac{|z_{k}|}{|z_{k+1}|} \leq \frac{\sqrt{\delta + \sum_{i=0}^{k}z_i^2}}{\sqrt{\delta + \sum_{i=0}^{k+1}z_i^2}} < 1,
\end{align*}
because for all $k \in \NN$, $z_{k+1}^2 > 0$. The proof is complete.

\section*{Acknowledgements}
Most of this work took place during the first author master internship at IRIT. His PhD at the University of Genova, is supported by the ITN-ETN project TraDE-OPT funded by the European Union’s Horizon 2020 research and innovation programme under the Marie Skłodowska-Curie grant agreement No 861137.

The second author would like to acknowledge the support of ANR-3IA Artificial and Natural Intelligence Toulouse Institute, Air Force Office of Scientific Research, Air Force Material Command, USAF, under grant numbers FA9550-19-1-7026, FA9550-18-1-0226, and ANR MaSDOL - 19-CE23-0017-01.

We warmly thank the anonymous referee for careful reading and relevant suggestions which improved the quality of the manuscript.
\bibliographystyle{abbrv} 
\bibliography{convex_case}

\begin{thebibliography}{10}

\bibitem{bach2012optimization}
F.~Bach, R.~Jenatton, J.~Mairal, and G.~Obozinski.
\newblock Optimization with sparsity-inducing penalties.
\newblock {\em Foundations and Trends{\textregistered} in Machine Learning},
  4(1):1--106, 2012.

\bibitem{bach2019universal}
F.~Bach and K.~Y. Levy.
\newblock A universal algorithm for variational inequalities adaptive to
  smoothness and noise.
\newblock {\em arXiv preprint arXiv:1902.01637}, 2019.

\bibitem{barakat2018convergence}
A.~Barakat and P.~Bianchi.
\newblock Convergence and dynamical behavior of the adam algorithm for non
  convex stochastic optimization.
\newblock {\em arXiv preprint arXiv:1810.02263}, 2018.

\bibitem{barakat2020convergence}
A.~Barakat and P.~Bianchi.
\newblock Convergence rates of a momentum algorithm with bounded adaptive step
  size for nonconvex optimization.
\newblock In {\em Asian Conference on Machine Learning}, pages 225--240. PMLR,
  2020.

\bibitem{bottou2018optimization}
L.~Bottou, F.~E. Curtis, and J.~Nocedal.
\newblock Optimization methods for large-scale machine learning.
\newblock {\em Siam Review}, 60(2):223--311, 2018.

\bibitem{chen2018sadagrad}
Z.~Chen, Y.~Xu, E.~Chen, and T.~Yang.
\newblock Sadagrad: Strongly adaptive stochastic gradient methods.
\newblock In {\em International Conference on Machine Learning}, pages
  913--921, 2018.

\bibitem{combettes2000fejer}
P.~L. Combettes.
\newblock {\em Fej{\'e}r monotonicity in convex optimization}, pages
  1016--1024.
\newblock Springer US, 2001.

\bibitem{combettes2001quasi}
P.~L. Combettes.
\newblock Quasi-fej{\'e}rian analysis of some optimization algorithms.
\newblock In {\em Studies in Computational Mathematics}, volume~8, pages
  115--152. Elsevier, 2001.

\bibitem{combettes2011proximal}
P.~L. Combettes and J.-C. Pesquet.
\newblock Proximal splitting methods in signal processing.
\newblock In {\em Fixed-point algorithms for inverse problems in science and
  engineering}, pages 185--212. Springer, 2011.

\bibitem{combettes2013variable}
P.~L. Combettes and B.~C. V{\~u}.
\newblock Variable metric quasi-fej{\'e}r monotonicity.
\newblock {\em Nonlinear Analysis: Theory, Methods \& Applications}, 78:17--31,
  2013.

\bibitem{defossez2020convergence}
A.~D{\'e}fossez, L.~Bottou, F.~Bach, and N.~Usunier.
\newblock On the convergence of adam and adagrad.
\newblock {\em arXiv preprint arXiv:2003.02395}, 2020.

\bibitem{duchi2011adaptive}
J.~Duchi, E.~Hazan, and Y.~Singer.
\newblock Adaptive subgradient methods for online learning and stochastic
  optimization.
\newblock {\em Journal of machine learning research}, 12(7), 2011.

\bibitem{goodfellow2016deep}
I.~Goodfellow, Y.~Bengio, A.~Courville, and Y.~Bengio.
\newblock {\em Deep learning}, volume~1.
\newblock MIT press Cambridge, 2016.

\bibitem{DBLP:journals/corr/KingmaB14}
D.~P. Kingma and J.~Ba.
\newblock Adam: {A} method for stochastic optimization.
\newblock In Y.~Bengio and Y.~LeCun, editors, {\em 3rd International Conference
  on Learning Representations, {ICLR} 2015, San Diego, CA, USA, May 7-9, 2015,
  Conference Track Proceedings}, 2015.

\bibitem{levy2018online}
K.~Y. Levy, A.~Yurtsever, and V.~Cevher.
\newblock Online adaptive methods, universality and acceleration.
\newblock In {\em Advances in Neural Information Processing Systems}, pages
  6500--6509, 2018.

\bibitem{li2019convergence}
X.~Li and F.~Orabona.
\newblock On the convergence of stochastic gradient descent with adaptive
  stepsizes.
\newblock In {\em The 22nd International Conference on Artificial Intelligence
  and Statistics}, pages 983--992. PMLR, 2019.

\bibitem{malitsky2019adaptive}
Y.~Malitsky and K.~Mishchenko.
\newblock Adaptive gradient descent without descent.
\newblock {\em arXiv preprint arXiv:1910.09529}, 2019.

\bibitem{mcmahan2010adaptive}
H.~B. McMahan and M.~Streeter.
\newblock Adaptive bound optimization for online convex optimization.
\newblock {\em arXiv preprint arXiv:1002.4908}, 2010.

\bibitem{nesterov1998introductory}
Y.~Nesterov.
\newblock Introductory lectures on convex programming volume i: Basic course.
\newblock {\em Lecture notes}, 3(4):5, 1998.

\bibitem{ogaltsov2019adaptive}
A.~Ogaltsov, D.~Dvinskikh, P.~Dvurechensky, A.~Gasnikov, and V.~Spokoiny.
\newblock Adaptive gradient descent for convex and non-convex stochastic
  optimization.
\newblock {\em arXiv preprint arXiv:1911.08380}, 2019.

\bibitem{ward2019adagrad}
R.~Ward, X.~Wu, and L.~Bottou.
\newblock Adagrad stepsizes: Sharp convergence over nonconvex landscapes.
\newblock In {\em International Conference on Machine Learning}, pages
  6677--6686, 2019.

\bibitem{wilson2017marginal}
A.~C. Wilson, R.~Roelofs, M.~Stern, N.~Srebro, and B.~Recht.
\newblock The marginal value of adaptive gradient methods in machine learning.
\newblock In {\em Advances in neural information processing systems}, pages
  4148--4158, 2017.

\bibitem{xie2020linear}
Y.~Xie, X.~Wu, and R.~Ward.
\newblock Linear convergence of adaptive stochastic gradient descent.
\newblock In {\em International Conference on Artificial Intelligence and
  Statistics}, pages 1475--1485. PMLR, 2020.

\end{thebibliography}

\end{document}